\documentclass[11pt, oneside]{amsart}

\usepackage[pdfusetitle,pdfauthor={Giuseppe Bargagnati, Federica Bertolotti, Pietro Capovilla and Francesco Milizia}]{hyperref}

\usepackage[english]{babel}

\usepackage{tikz}
\usepackage{comment}
\usepackage{url}
\usepackage{caption}
\usepackage[final]{microtype}

\usepackage{tikz-cd}

\usepackage{amsmath}
\usepackage{amssymb} 
\usepackage{amsthm}
\usepackage[textsize=tiny]{todonotes}
\setuptodonotes{fancyline, color=blue!30}

\newtheorem{lemma}{Lemma}[section]

\newtheorem{prop_intro}{Proposition}

\newtheorem{quest_intro}[prop_intro]{Question}
\newtheorem{thm_intro}[prop_intro]{Theorem}

\newtheorem{cor_intro}[prop_intro]{Corollary}

\makeatletter
\newtheorem*{var@theorem}{\var@title}
\newcommand{\newvartheorem}[2]{%
	\newenvironment{var#1}[1]{%
		\def\var@title{#2 \ref{##1}}%
		\begin{var@theorem}}%
			{\end{var@theorem}}}
\makeatother
\newvartheorem{thm}{Theorem}

\theoremstyle{definition}

\newtheorem{defn}[lemma]{Definition}

\newtheorem{rmk}[lemma]{Remark}

\newtheorem*{prop*}{Proposition}

\DeclareMathOperator\out{Out}

\DeclareMathOperator\inn{Inn}
\DeclareMathOperator{\aut}{{\mathrm{Aut}}}
\DeclareMathOperator\isom{Isom}
\DeclareMathOperator\mcg{MCG}

\renewcommand{\phi}{\varphi}

\newcommand{\calC}{{C\hspace{-0.15em}H}_b^2(\Gamma)}
\newcommand{\Hom}{\mathrm{Hom}}
\newcommand{\str}{\mathrm{str}}

\newcommand{\N}{\ensuremath {\mathbb{N}}}
\newcommand{\R}{\ensuremath {\mathbb{R}}}

\newcommand{\Z}{\ensuremath {\mathbb{Z}}}

\newcommand{\Hp} {\ensuremath {\mathbb{H}^2}}
\newcommand{\bi} {{\partial_\infty}}
\newcommand{\acts}{\curvearrowright}
\newcommand\restr[2]{{\left.\kern-\nulldelimiterspace #1 \vphantom{\big|} \right|_{#2} }}
\newcommand{\eqs}{{\tilde{\tau}_\sigma}}
\newcommand{\eqsb}{{\eta}}

\title[]{The action of mapping class groups on de Rham quasimorphisms}
\keywords{}

\author[]{Giuseppe Bargagnati}
\address{Dipartimento di Matematica, Università di Pisa, Italy}
\email{giuseppe.bargagnati@phd.unipi.it}

\author[]{Federica Bertolotti}
\address{Scuola Normale Superiore, Pisa, Italy}
\email{federica.bertolotti@sns.it}

\author[]{Pietro Capovilla}
\address{Scuola Normale Superiore, Pisa, Italy}
\email{pietro.capovilla@sns.it}

\author[]{Francesco Milizia}
\address{Scuola Normale Superiore, Pisa, Italy}
\email{francesco.milizia@sns.it}

\begin{document}

\begin{abstract}
	We study the action of the mapping class group on the subspace of de Rham classes in the degree-two bounded cohomology of a hyperbolic surface. In particular, we show that the only fixed nontrivial finite-dimensional subspace is the one generated by the Euler class.
	As a consequence, we get that the action of the mapping class group on the space of de Rham quasimorphisms has no fixed points.
\end{abstract}

\maketitle

\section{Introduction}


Let $\Gamma$ be a group. A quasimorphism on $\Gamma$ is a map $f:\Gamma \rightarrow \mathbb{R}$ for which there is a $D\in \R_{\ge 0}$ such that
\[|f(g)+f(h)-f(gh)|\leq D\]
for every $g,h \in \Gamma$.
A quasimorphism is \emph{homogeneous} if it restricts to a homomorphism on every cyclic subgroup of $\Gamma$.
The first examples of nontrivial quasimorphisms were constructed for free groups by Brooks \cite{Bro81}. Notice that a quasimorphism is considered \emph{trivial} if it is at bounded distance from a homomorphism.

The group of automorphisms of a group acts naturally on the space of (homogeneous) quasimorphisms by precomposition. In 2010, Mikl\'os Ab\'ert asked whether for $n \in \N_{\geq 2}$ there exist nonzero $\aut$-invariant homogeneous quasimorphisms on the free group with $n$ generators $F_n$; he comments ``probably not'' \cite[Question 47]{Abe10}. In \cite{Has18} Hase proved that the action of $\aut(F_n)$ on the space of Brooks quasimorphisms (which is invariant under the action of $\aut(F_n)$) has no nontrivial fixed points. On the other hand, in \cite{BB19} Brandenbursky and Marcinkowski answered this question affirmatively for $n=2$. Finally, in \cite{FFW23} Fournier-Facio and Wade proved that for a large class of groups (which includes non-elementary Gromov hyperbolic groups and infinitely-many-ended groups) there exists an infinite-dimensional space of $\aut$-invariant homogeneous quasimorphisms.
However, their construction is not very explicit;
therefore, it makes sense to restrict to some particular $\aut$-invariant subspaces of quasimorphisms --- coming, for example, from geometric or combinatorial constructions --- and ask whether they contain nontrivial fixed points or not.

We focus our attention on \emph{de Rham quasimorphisms}, which are considered by Calegari in \cite[Subsection 4.2]{scl} and defined as follows.
Let $\Sigma$ be an oriented closed connected surface of genus at least two and denote by $\Gamma$ its fundamental group.
Let $\alpha \in \Omega^1(\Sigma)$ be a $1$-form and $m$ be a hyperbolic metric on $\Sigma$.
For every $\gamma \in \Gamma\setminus\{1\}$, we denote by $\rho^m_{\gamma}$ the \emph{free} oriented closed geodesic (with respect to the metric $m$) in the homotopy class of $\gamma$. Then, the (homogeneous) de Rham quasimorphism associated to $m$ and $\alpha$ is defined by setting
\[q^m_\alpha(\gamma)= \int_{\rho^m_{\gamma}}\alpha\]
for every $\gamma \in \Gamma\setminus\{1\}$, and $q^m_\alpha(1) = 0$.
A quasimorphism on $\Gamma$ is a de Rham quasimorphism if it is of the form $q^m_\alpha$ for some $m$ and $\alpha$ as above.

The vector space generated by de Rham quasimorphisms is $\aut(\Gamma)$\nobreakdash-\hspace{0pt}invariant (see Section \ref{sec:qm vs bc}).
We prove the following result.
\begin{thm_intro}
	\label{thm: invariant_quasi_2}
	The action of $\aut(\Gamma)$ on the space of quasimorphisms generated by de Rham quasimorphisms has just one nonzero finite-\hspace{0pt}dimensional invariant subspace, that is the one consisting of homomorphisms.
\end{thm_intro}

Notice that $\Hom(\Gamma,\R)$ has dimension $2g$, where $g$ is the genus of the surface.
In particular, since there are no $1$-dimensional $\aut(\Gamma)$-invariant subspaces, we deduce the following corollary.
\begin{cor_intro}
	\label{cor: invariant continuous quasimorphismsintro}
	The action of $\aut(\Gamma)$ on the space of quasimorphisms generated by de Rham quasimorphisms has no nonzero fixed points.
\end{cor_intro}
\begin{rmk}
	Our results should be read as the analogue for surface groups of Hase's ones \cite{Has18} for free groups.
\end{rmk}
\begin{rmk}
	The result by Fournier-Facio and Wade (\cite[Theorem B]{FFW23}), together with Corollary \ref{cor: invariant continuous quasimorphismsintro}, readily implies that there is an infinite-dimensional subspace of homogeneus quasimorphisms that cannot be described as (a linear combination of) de Rham quasimorphisms. 
\end{rmk}
In order to establish Theorem \ref{thm: invariant_quasi_2} we adopt the point of view of bounded cohomology, which in degree two is strictly related to quasimorphisms.
Bounded cohomology of groups was introduced by Johnson in \cite{Joh72}, where he proved that it vanishes in all positive degrees for amenable groups. Pioneering works in this area were those of Gromov \cite{Gro82} and Ivanov \cite{Iva87}, where bounded cohomology of topological spaces was also introduced. Since then, it has played a fundamental role in various fields of mathematics, such as geometric group theory \cite{Min02}, simplicial volume \cite{Gro82, Thurston79}, circle actions \cite{Ghy01}, rigidity theory \cite{burger2002continuous} and stable commutator length \cite{scl}.

The connection with quasimorphisms comes from the fact that the space of homogeneous quasimorphisms on $\Gamma$ modulo homomorphisms canonically injects into the second bounded cohomology $H^2_b(\Gamma)$.
The image of this map coincides with the set of bounded classes that vanish in standard cohomology, which are also called \emph{exact bounded cohomology classes} (e.g., \cite[Section 2.3]{frigerio2017bounded}).
The group $\aut(\Gamma)$ acts naturally on $H^2_b(\Gamma)$, and the just mentioned injection is equivariant with respect to this action.

From this point of view, de Rham quasimorphisms correspond to \emph{de Rham classes}, which are bounded cohomology classes defined via differential $2$-forms and hyperbolic metrics on $\Sigma$: in 1988, Barge and Ghys \cite{bargeghys} proved that, given a hyperbolic metric $m$ on $\Sigma$, the second bounded cohomology of $\Sigma$ contains an infinite-dimensional subspace given by differential 2-forms on $\Sigma$. Namely, they showed that the map
\[\Psi_m: \Omega^2(\Sigma)\rightarrow H^2_b(\Sigma),\]
defined by integrating differential forms over geodesic triangles (with respect to the metric $m$), is injective (see Section \ref{Sec:BCDF} for the precise definition).
We call \emph{de Rham classes} the bounded cohomology classes of $H^2_b(\Sigma)$ (or, equivalently, of $H^2_b(\Gamma)$) which lie in the image of $\Psi_m$ for some hyperbolic metric $m$.
These classes have been intensively studied, also in recent years \cite{marasco2022trivial, BFMSS22, MB22, Mar23}.
We refer to Section \ref{sec:qm vs bc} for a detailed description of the relationship between de Rham quasimorphisms and de Rham classes.
Notice that the Euler class of $\Sigma$ is a multiple of the de Rham class corresponding to any metric and its corresponding volume form.


Since we are dealing with a geometric construction, it is helpful to rethink the action of $\aut(\Gamma)$ in more geometric terms.
First, recall that inner automorphisms act trivially on (bounded) cohomology \cite[Section II.6]{browncohomology}.
Therefore, the action of $\aut(\Gamma)$ factors through an action of $\out(\Gamma) = \aut(\Gamma)/\inn(\Gamma)$.
On the other hand, 
a celebrated theorem by Dehn, Nielsen and Baer affirms that the (positive) mapping class group, denoted by $\mcg(\Sigma)$, is naturally isomorphic to a subgroup of index 2 of $\out(\Gamma)$.
Therefore, $\mcg(\Sigma)$ acts on bounded cohomology.
Go to Section \ref{Sec:mcgonH2b} for a description of this action, which clarifies also why the subspace of $H^2_b(\Sigma)\cong H^2_b(\Gamma)$ generated by the de Rham classes is invariant under the action of the mapping class group.

We prove the following result.
\begin{thm_intro}
	\label{thm: invariant de Rham}
	The subspace of $H^2_b(\Gamma)$ generated by de Rham classes has only two finite-dimensional subspaces preserved by $\mcg(\Sigma)$: the trivial one and the $1$-dimensional subspace generated by the Euler class.
\end{thm_intro}

In order to prove the previous theorem, we study the action of $\mcg(\Sigma)$ on a bigger subspace of $H^2_b(\Gamma)$, which we denote by $\calC$; it is still $\mcg(\Sigma)$-invariant and contains all the de Rham classes (see Subsection \ref{sec: bc as meas fun} for the definition).
In this context, we prove the following theorem, of which Theorem \ref{thm: invariant de Rham} is an immediate corollary.

\begin{thm_intro}
	\label{thm: invariant continuous subspacesintro}
	The subspace generated by the Euler class is the only finite-dimensional nontrivial subspace preserved by the action $\mcg(\Sigma) \acts \calC$.
\end{thm_intro}
\begin{rmk}
	By translating \cite[Theorem B]{FFW23} in the bounded cohomology framework, one obtains that $H^2_b (\Sigma)$ has an infinite-dimensional subspace consisting of points fixed by the action of $\mcg(\Sigma)$. In particular, the map $\Psi_m$ is \emph{highly nonsurjective}, in the sense that there is an infinite-dimensional subspace of $H_b^2(\Sigma)$ not contained in the image of $\Psi_m$.
\end{rmk}
It is not clear to us how the image $\Psi_m(\Omega^2(\Sigma))$ depends on the chosen hyperbolic metric $m$.
In particular, even if the space of de Rham classes (varying the hyperbolic metric) is $\mcg(\Sigma)$-invariant, we do not know if, keeping fixed the hyperbolic metric $m$, the subspace $\Psi_m(\Omega^2(\Sigma))$ is invariant under the action of $\mcg(\Sigma)$.

\begin{quest_intro}
	\label{que: dependence on metric}
	How does the subspace of $H^2_b(\Sigma)$ given by $\Psi_m(\Omega^2(\Sigma))$ depend on the chosen hyperbolic metric $m$ on $\Sigma$?
\end{quest_intro}

\begin{quest_intro}
	\label{que: invarinace on metric}
	Let $m$ be a hyperbolic metric on $\Sigma$. Is the subspace $\Psi_m(\Omega^2(\Sigma))$ of $H^2_b(\Gamma)$ invariant under the action of $\mcg(\Sigma)$?
\end{quest_intro}

\subsection*{Plan of the paper}
In Section \ref{Sec:BCDF}, we recall the definition of bounded cohomology of groups and spaces and define de Rham classes following \cite{bargeghys}. Then we introduce the space $\calC$ and show that it contains every de Rham class. In Section \ref{Sec:mcgonH2b} we describe how the action of the mapping class group on $H^2_b(\Sigma)$ descends to an action on $\calC$, focusing in particular on the action of Dehn twists. We devote Section \ref{Sec:proof} to the proof of Theorem \ref{thm: invariant continuous subspacesintro}. Finally, in Section \ref{sec:qm vs bc} we explain the relationship between de Rham quasimorphisms and de Rham classes and we prove Theorem \ref{thm: invariant_quasi_2} and Corollary \ref{cor: invariant continuous quasimorphismsintro}.

\subsection*{Acknowledgments}
We are grateful to our supervisor Roberto Frigerio  for helpful discussions.
We thank Francesco Fournier-Facio for suggesting us this problem.
We thank Francesco Bonsante, Gabriele Mondello and Maria Beatrice Pozzetti for useful comments.

Moreover, G.B., F.B.\ and F.M.\ have been supported by the INdAM GNSAGA Project, CUP E55F22000270001.

\section{Bounded cohomology and differential forms}\label{Sec:BCDF}

\subsection{Bounded cohomology of groups}
Let $\Gamma$ be a discrete group.
The \emph{real bound\-ed cohomology} of $\Gamma$, denoted by $H^\bullet_b(\Gamma)$, is defined as the cohomology of the following complex of vector spaces:
\[
	0 \rightarrow C^0_{b}(\Gamma)^\Gamma \rightarrow C^1_{b}(\Gamma)^\Gamma \rightarrow C^2_{b}(\Gamma)^\Gamma \rightarrow \cdots\,,
\]
where $C^n_{b}(\Gamma)^\Gamma$ denotes the space of set-theoretic bounded $\Gamma$-invariant maps from $\Gamma^{n+1}$ to $\R$, and the differential maps are defined by $(\delta f)(\gamma_0,\dots,\gamma_n) = \sum_{i=0}^n(-1)^if(\dots,\hat{\gamma_i},\dots)$.
Here, the action of $\Gamma$ on functions is defined by $(\gamma\cdot f)(\gamma_0,\dots,\gamma_n)=f(\gamma^{-1}\gamma_0,\dots, \gamma^{-1}\gamma_n)$.

The group of automorphisms of $\Gamma$ naturally acts of $C^n_b(\Gamma)$: for every $\varphi \in \aut(\Gamma)$ and every $f \in C^n_b(\Gamma)$ we set
\[
	(\varphi \cdot f) (g_0,\dots,g_n) = f(\varphi^{-1}(g_0),\dots, \varphi^{-1}(g_n)).
\]
Since $\aut(\Gamma)$ acts on $C^\bullet_b(\Gamma)$ by $\Gamma$-equivariant chain maps, we get a well-defined action $\aut(\Gamma)\acts H^\bullet_b(\Gamma)$.
Moreover, being trivial on inner automorphisms \cite[Section II.6]{browncohomology}, this action factors through $\out(\Gamma)$.

\subsection{Bounded cohomology of spaces}
Let $X$ be a topological space. The \emph{real bounded cohomology} $H^\bullet_b(X)$ of $X$ is the cohomology of the complex
\[
	0\rightarrow C^0_b(X)\rightarrow C^1_b(X)\rightarrow C^2_b(X)\rightarrow\cdots \,,
\]
where $C^n_{b}(X)$ denotes the space of bounded singular cochains of $X$, the differential maps are the restrictions of the boundary maps of the usual singular cochain complex $C^\bullet(X)$ and a cochain $f \in C^n(X)$ is called \emph{bounded} if
\[
	\lVert f \rVert_\infty =
	\sup \bigl\{ | f(\sigma) |, \; \sigma \mbox{ is a singular $n$-simplex}
	\bigr\} < \infty .
\]

\subsection{Bounded cohomology and differential forms}\label{sub:derham clas}
From now on, $\Sigma$ will be a closed oriented surface of genus at least two.
We fix a base point $p \in \Sigma$ and let $\Gamma =\pi_1(\Sigma,p)$ be the fundamental group.
We also fix a base point $\tilde p$ in the universal cover of $\Sigma$, projecting to $p\in\Sigma$ via the universal covering map, so that $\Gamma$ can be identified to the group of deck transformations.

Let $m$ be a hyperbolic metric on $\Sigma$; the universal cover, endowed with the pull-back metric, can be isometrically identified with the hyperbolic plane $\Hp$ and, accordingly, $\Gamma$ can be seen as a subgroup of $\isom(\Hp)$.
The hyperbolic metric allows us to perform a straightening procedure of singular simplices: 
if $\tilde{s}$ is a singular simplex of $\Hp$, its straightening $\str(\tilde{s})$ is a \emph{geodesic} singular simplex in $\Hp$ having the same vertices as $\tilde{s}$ (for the details see, e.g., \cite[Section 6.1]{Thurston79} and \cite[Section 8.4]{frigerio2017bounded}).

Let $\Omega^2(\Sigma)$ be the space of differential $2$-forms on $\Sigma$.
Every $\omega \in \Omega^2(\Sigma)$ defines a singular $2$-cochain $c_\omega \in C^2(\Sigma)$ in the following way: for every singular $2$-simplex $s$, $c_\omega(s)$ is the integral of $\tilde{\omega}$ on the straightening of any lift of $s$, where $\tilde{\omega}$ denotes the pull-back of $\omega$ to the universal cover.

Since the area of geodesic triangles in $\Hp$ is bounded from above by $\pi$, $c_\omega$ actually defines a \emph{bounded} singular $2$-cochain $c_\omega \in C^2_b(\Sigma)$.
It was proven by Barge and Ghys in \cite{bargeghys} that $c_\omega$ is actually a cocycle and that the map
\[\begin{aligned}
		\Psi_m \colon \Omega^2(\Sigma)\rightarrow & H^2_b(\Sigma) \\ \omega \mapsto &[c_\omega]
	\end{aligned}\]
is injective.
We are keeping the metric $m$ as a subscript, since it plays an important role in the definition of the map.

\begin{defn}
	We call \emph{de Rham classes} the cohomology classes in $H^2_b(\Sigma)$ of the form $\Psi_m(\omega)$, for some hyperbolic metric $m$ and $\omega \in \Omega^2(\Sigma)$.
\end{defn}
Consider the map $\theta \colon C^2_b(\Sigma) \rightarrow C^2_b(\Gamma)^\Gamma$ defined by the formula
\[(\theta(f))(\gamma_0,\gamma_1,\gamma_2)=\tilde{f}(\str(\gamma_0\cdot \tilde p,\gamma_1\cdot \tilde p,\gamma_2\cdot \tilde p)),\]
where $\tilde{f} \in C^2_b(\Hp)^\Gamma\cong C^2_b(\Sigma)$ denotes the $\Gamma$-invariant lift of $f\in C^2_b(\Sigma)$ to $\Hp$, and $\str(\gamma_0\cdot \tilde p,\gamma_1\cdot \tilde p,\gamma_2\cdot \tilde p)$ denotes the geodesic singular triangle of $\Hp$ with vertices $\gamma_0\cdot \tilde p,\gamma_1\cdot \tilde p,\gamma_2\cdot \tilde p \in \Hp$.
It is well known that $\theta$ induces an isomorphism in bounded cohomology and that this isomorphism is independent of the chosen metric (see \cite[Corollary 4.15]{frigerio2017bounded}); for this reason, this time we are omitting $m$ in the notation (we denote by $\theta$ also the isomorphism in cohomology).
The classes in $H^2_b(\Gamma)$ corresponding to the de Rham classes via the isomorphism $\theta$ will be also called \emph{de Rham classes}.


The linear subspace $\theta(\Psi_m(\Omega_2(\Sigma))) \subseteq H^2_b(\Gamma)$ can also be described, as done by Barge and Ghys, by integrating 2-forms on ideal triangles in $\Hp$.
Recall that, as $\Gamma$ acts by isometries on $\Hp$, it also acts on the boundary at infinity $\bi \Hp \cong S^1$ by homeomorphisms.
Fix $\xi \in \bi \Hp$ and, for $\omega \in \Omega^2(\Sigma)$, let $\tilde{\omega} \in \Omega^2(\Hp)$ be the pull-back of $\omega$ to $\Hp$.
Define $c_{\omega, \xi} \in C^2_b(\Gamma)^\Gamma$ by setting
\begin{equation}\label{eq:cociclo al bordo}
	c_{\omega,\xi}(\gamma_0,\gamma_1,\gamma_2)=
	\int_{\str(\gamma_0\cdot \xi, \gamma_1\cdot \xi, \gamma_2\cdot \xi)}\tilde \omega\
\end{equation}
for every $\gamma_0,\gamma_1,\gamma_2\in \Gamma$, where $\str(\xi_0, \xi_1, \xi_2)\subset\Hp$ denotes the oriented ideal triangle with vertices $\xi_0, \xi_1, \xi_2 \in \bi \Hp$. Here, by oriented we mean that the ideal simplex $\str(\xi_0, \xi_1, \xi_2)$ is furnished with the same (resp.\ opposite) orientation of $\Hp$ whenever the triple $(\xi_0, \xi_1, \xi_2)$ is ordered counter-clockwise (resp.\ clockwise) on $S^1$. 
If $\gamma_i \cdot \xi =\gamma_j \cdot \xi$ for some $i\neq j$, then the ideal simplex is degenerate, thus $c_{\omega,\xi}(\gamma_0,\gamma_1,\gamma_2)= 0$.

In the following lemma, the cochains $\theta(c_\omega)$ and $c_{\omega, \xi}$ in $C^2_b(\Gamma)^\Gamma$ are defined using the same hyperbolic metric $m$.
\begin{lemma}[{\cite[Lemma 3.10]{bargeghys}}]
	\label{lem: barge and ghys}
	For every $\omega \in \Omega^2(\Sigma)$ and $\xi \in \bi\Hp$, the cochain $\theta(c_\omega)$ is cohomologous to $c_{\omega, \xi}$, therefore they define the same class in $H^2_b(\Gamma)$.
\end{lemma}

\subsection{Bounded cohomology 2-classes as measurable functions}\label{sec: bc as meas fun}
Let now $m_0$ be a hyperbolic metric on $\Sigma$, and choose an isometry between the universal cover of $\Sigma$ and $\Hp$, so that $\Gamma$ acts on $\bi \Hp \cong S^1$ by homeomorphisms.
We are using a different notation for $m_0$, since this metric will serve a different purpose compared to the metric $m$ of the previous subsection.
If we equip $S^1$ with the Lebesgue probability measure (on the $\sigma$-algebra of Borel subsets)
the action $\Gamma\curvearrowright S^1$ is doubly ergodic \cite[Example 4]{burger2002continuous}. Hence, we have an identification \cite[Theorem 2]{burger2002continuous}
\begin{equation}
	\label{eq: bc is measurable functions}
	H_b^2(\Gamma) \cong ZL^\infty_{\rm alt}\left(S^1\times S^1 \times S^1 \right)^\Gamma,
\end{equation}
where $ZL^\infty_{\rm alt}\left(S^1\times S^1 \times S^1 \right)$ is the space of $\Gamma$-invariant alternating measurable
bounded cocycles (up to equality almost everywhere) from $S^1\times S^1 \times S^1$ to $\R$. Here, by alternating we mean that if we permute the entries of the cocycle, then the sign of the output changes according to the sign of the permutation.
Again, the superscript $\Gamma$ denotes taking the space of $\Gamma$-invariant cochains with respect to the diagonal action on $S^1 \times S^1 \times S^1$.

\begin{rmk}\label{rmk:f_classes}
	The elements of $ZL^\infty_{\rm alt}\left(S^1\times S^1 \times S^1 \right)$ are not functions, but \emph{classes} of functions.
	This being said, we will occasionally also write $f \in ZL^\infty_{\rm alt}\left(S^1\times S^1 \times S^1 \right)^\Gamma$ with $f$ a function, meaning that it represents a class that belongs to $ZL^\infty_{\rm alt}\left(S^1\times S^1 \times S^1 \right)^\Gamma$.
	Notice that such an $f$ is not necessarily $\Gamma$-invariant as a function, even if its class is, and the cocycle condition $\delta f = 0$, in general, only holds almost everywhere.
\end{rmk}

\begin{defn}
	Let $\Delta \subset S^1 \times S^1 \times S^1$ be the multidiagonal, consisting of triples in which at least two points coincide.
	We denote by $\calC$ the linear subspace of $H_b^2(\Gamma) \cong ZL^\infty_{\rm alt}\left(S^1\times S^1 \times S^1 \right)^\Gamma$ whose elements can be represented by cocycles which are continuous on $(S^1 \times S^1 \times S^1)\setminus\Delta$.
\end{defn}

Since 
$\Delta$ has measure $0$, and since two continuous functions on $(S^1)^3\setminus \Delta$ are equal if and only if they coincide almost everywhere, the space $\calC$ can be identified with
\[\{f \colon (S^1)^3\setminus\Delta \to \R \text{ alternating } \Gamma \text{-invariant continuous bounded cocycle}\}.\]

An example of an element inside $\calC$ is the \emph{orientation cocycle} 
\[\mathrm{Or} \colon (S^1)^3\setminus\Delta \to \{-1,+1 \} ,\]
which assigns $+1$ to every positively oriented triple (i.e., ordered counter-clockwise) and $-1$ to every negatively oriented triple.
Being a multiple of the orientation cocycle, also the \emph{Euler class} of $\Sigma$ lies in $\calC$.

Since the action of $\Gamma$ on $\Hp$ (and therefore on $\bi \Hp\cong S^1$) depends on the chosen hyperbolic metric $m_0$ and on the chosen isometry between $\Hp$ and the universal cover of $\Sigma$, the space $\calC$ could also \emph{a priori} depend on this data.
In the following lemma, we show that this is not the case.

\begin{lemma}
	\label{lem: indipendence of calC from metric}
	The space $\calC$, understood as a linear subspace of $H^2_b(\Gamma)$, is independent of the choice of the metric $m_0$ and the isometry between $\Hp$ and the universal cover of $\Sigma$.
\end{lemma}
\begin{proof}
	Consider two hyperbolic metrics $m_0$ and $\overline{m}_0$ on $\Sigma$, and two isometries from the universal cover of $\Sigma$ (endowed with the lifts of $m_0$ and of $\overline{m}_0$) to $\Hp$. By the fundamental property of universal coverings, we have a $\Gamma$-equivariant isometry $h\colon \Hp \to \Hp$, where the actions of $\Gamma$ on the domain and the codomain depend on the two identifications corresponding to $m_0$ and $\overline{m}_0$.
	We denote by $\partial h\colon S^1\rightarrow S^1$ the $\Gamma$-equivariant homeomorphism induced by $h$ on the boundary at infinity.
	We have two resolutions
	\[ 0 \to \R \to L^\infty_{\rm alt}\left(S^1\right)\to L^\infty_{\rm alt}\left(S^1\times S^1\right) \to L^\infty_{\rm alt}\left(S^1\times S^1 \times S^1\right) \to \dots \]
	of $\R$ by normed $\Gamma$-modules --- what changes between the two is the action of $\Gamma$.
	However, they are both strong resolutions by relatively injective $\Gamma$-modules; therefore, any $\Gamma$-equivariant chain map from them to the standard resolution $C_b^\bullet(\Gamma)$, extending the identity on $\R$, induces a canonical isomorphism in cohomology, independent of the chain map \cite[Corollary 4.15]{frigerio2017bounded}.
	For example, we consider the chain maps $\psi_{m_0}^\bullet\colon L^\infty_{\rm alt}\left((S^1)^{\bullet+1}\right) \rightarrow C^\bullet_b(\Gamma)$ and $\psi_{\overline{m}_0}^\bullet\colon L^\infty_{\rm alt}\left((S^1)^{\bullet+1}\right) \rightarrow C^\bullet_b(\Gamma)$ defined by the following formula:
	\begin{equation}
		\label{eq: chain map psi}
		\psi^n(f)(\gamma_0,\dots,\gamma_n)= \int_{(S^1)^{n+1}}f(\gamma_0\cdot x_0,\dots, \gamma_n\cdot x_n) \ d\mu(x_0,\dots,x_n) ,
	\end{equation}
	where $\mu$ denotes the Lebesgue probability measure on $(S^1)^{n+1}$. Again, what changes between them is the action of $\Gamma$ on $S^1$. Since they are $\Gamma$-equivariant chain maps, they both induce the isomorphism in (\ref{eq: bc is measurable functions}).
	Let $F_{m_0} \in ZL^\infty_{\rm alt}\left((S^1)^3\right)^\Gamma$ be a $\Gamma$-invariant cocycle admitting a continuous representative $f_{m_0}\colon (S^1)^3\setminus \Delta \rightarrow \R$.
	We want to show that $\psi_{m_0}^2(F_{m_0})$ is cohomologous in $C^\bullet_b(\Gamma)^\Gamma$ to some cocycle of the form $\psi_{\overline{m}_0}^2(F_{\overline{m}_0})$, where $F_{\overline{m}_0}$ is a $\Gamma$-invariant cocycle admitting a continuous representative $f_{\overline{m}_0} \colon (S^1)^3\setminus \Delta\rightarrow \R$.
	To this end, we can simply define $f_{\overline{m}_0}$ as the push-forward of $f_{m_0}$ via $\partial h$:
	\[
	f_{\overline{m}_0}(\xi_0,\xi_1,\xi_2)=f_{m_0}(\partial h^{-1}(\xi_0), \partial h^{-1}(\xi_1), \partial h^{-1}(\xi_2)).
	\]
	Since $\partial h^{-1}$ is continuous and $\Gamma$-equivariant, it follows that $f_{\overline{m}_0}$ is $\Gamma$-invariant and continuous on $(S^1)^3\setminus\Delta$. 
	Moreover, as $\delta f_{m_0}\equiv 0$ (since $f_{m_0}$ is a continuous cocycle) and the map $(\xi_0,\xi_1,\xi_2)\mapsto (\partial h^{-1}(\xi_0), \partial h^{-1}(\xi_1), \partial h^{-1}(\xi_2))$ preserves the multidiagonal $\Delta$, it follows that $\delta f_{\overline{m }_0}\equiv 0$. 
	It remains to show that $\psi_{m_0}^2(F_{m_0})$ and $\psi_{\overline{m}_0}^2(F_{\overline{m}_0})$ are cohomologous in $C^\bullet_b(\Gamma)^\Gamma$.
	We have that
	\begin{align*}
		& (\psi_{m_0}^2(F_{m_0}) - \psi_{\overline{m}_0}^2(F_{\overline{m}_0}))(\gamma_0,\gamma_1,\gamma_2)\\
		& = \int_{(S^1)^{3}} f_{m_0}(\gamma_0 x_0,\dots, \gamma_2 x_2) \ d\mu
		-\int_{(S^1)^{3}} f_{\overline{m}_0}(\gamma_0 x_0,\dots, \gamma_2 x_2) \ d\mu \\
		& = \int_{(S^1)^{3}} \Big(f_{m_0}(\gamma_0 x_0,\dots, \gamma_2 x_2) - f_{m_0}\big(\partial h^{-1}(\gamma_0 x_0),\dots, \partial h^{-1}(\gamma_2 x_2)\big)\Big) \ d \mu  . 
	\end{align*}
	Therefore, if we set 
	\begin{align*}
	\beta(\gamma_0,\gamma_1)=
	\int_{(S^1)^{3}} \Big(&f_{m_0}\big(\gamma_0x_0, \gamma_1x_1, \partial h^{-1} (\gamma_1x_1)\big) \\
	&- f_{m_0}\big(\gamma_0x_0, \partial h^{-1} (\gamma_0 x_0), \partial h^{-1} (\gamma_1x_1)\big)\Big) \ d\mu ,
\end{align*}
it follows from the fact that $f_{m_0}$ is a cocycle that $\beta \in C^1_b(\Gamma)$ is a $\Gamma$-equivariant bounded cochain such that $\psi_{m_0}^2(F_{m_0}) - \psi_{\overline{m}_0}^2(F_{\overline{m}_0}) = \delta \beta$.
\end{proof}

\begin{lemma}
	\label{lem: deRham qm are continuous}
	For every hyperbolic metric $m$ on the surface $\Sigma$, the subspace $\Psi_m(\Omega^2(\Sigma)) \subseteq H^2_b(\Sigma) \cong H^2_b(\Gamma)$ is contained in $\calC$.
\end{lemma}
\begin{proof}
	By Lemma \ref{lem: indipendence of calC from metric}, we can assume that the metric $m$ used to inject $\Omega^2(\Sigma)$ into $H_b^2(\Gamma)$ coincides with the one that gives the inclusion of $\calC$ into $H_b^2(\Gamma)$.
	We know that the isomorphism in (\ref{eq: bc is measurable functions}) is induced by any $\Gamma$-equivariant chain map between the resolutions $C^\bullet_b(\Gamma)$ and $L^\infty_{\rm alt}\left((S^1)^\bullet\right)$ \cite[Corollary 4.15]{frigerio2017bounded}, in particular it is induced by the chain map $\psi\colon L^\infty_{\rm alt}\left((S^1)^{\bullet+1}\right) \rightarrow C^\bullet_b(\Gamma)$ defined in (\ref{eq: chain map psi}).
	For every $\omega \in \Omega^2(\Sigma)$ we define $f_{\omega} \in ZL^\infty_{\rm alt}(S^1\times S^1 \times S^1)^\Gamma$ as
	\[
		f_\omega(\xi_0,\xi_1,\xi_2)=\int_{\str(\xi_0,\xi_1,\xi_2)}\tilde{\omega} .
	\]
	Since $\tilde{\omega}\in \Omega^2(\Hp)$ is continuous and bounded, it is straightforward to check that the cocycle $f_\omega$ is continuous on $(S^1)^3 \setminus \Delta$, that is, $f_\omega \in \calC$.

	The cocycle $c_{\omega,\xi}$, defined in (\ref{eq:cociclo al bordo}), is cohomologous to $\psi(f_\omega)$: this can be seen by using the same prism construction as in \cite[proof of Lemma 3.9]{bargeghys}.
	Thus, putting together these facts with Lemma \ref{lem: barge and ghys}, one gets that $\psi(f_\omega) \in \calC$ is cohomologous to the de Rham class $\theta(c_\omega)$.
\end{proof}

\section{Action of the mapping class group on the second bounded cohomology}\label{Sec:mcgonH2b}
The (positive) mapping class group $\mcg(\Sigma)$ of $\Sigma$ is the group of orientation\nobreakdash-\hspace{0pt}preserving self-diffeomorphisms of $\Sigma$ considered up to isotopy.
Since homotopic continuous functions induce equal maps in bounded cohomology, we have a natural action of $\mcg(\Sigma)$ on $H_b^2(\Sigma)$.

\begin{rmk}
	\label{rmk: invariance de Rham classes}
	The linear subspace of $H^2_b(\Sigma)$ generated by de Rham classes is invariant under the action of the mapping class group. 
	Indeed, consider $\varphi \in \mcg(\Sigma)$ and let $h$ denote a self-diffeomorphism of $\Sigma$ that represents $\varphi$.
	Recall that $h$ acts on $\Omega^2(\Sigma)$ and on $\mathrm{Hyp}(\Sigma)$ by pull-backs, where $\mathrm{Hyp}(\Sigma)$ denotes the space of hyperbolic metrics on $\Sigma$. 
	Now, if we have a de Rham class $\Psi_m(\omega) \in H^2_b(\Sigma)$, where $m\in\mathrm{Hyp}(\Sigma)$ and $\omega \in \Omega^2(\Sigma)$, it's straightforward to verify that
	\[
	\varphi \cdot \Psi_m(\omega) = \Psi_{h\cdot m}(h^{-1}\cdot \omega) .
	\]
	Therefore $\varphi \cdot \Psi_m(\omega)$ is still a de Rham class. 
	Notice that some degree of flexibility  with the metrics is required in order to prove our claim (see Question \ref{que: dependence on metric} and Question \ref{que: invarinace on metric}).
\end{rmk}

Recall from Subsections \ref{sub:derham clas} and \ref{sec: bc as meas fun} that we have isomorphisms $H_b^2(\Sigma) \cong H_b^2(\Gamma) \cong ZL^\infty_{\rm alt} ((S^1)^3)^\Gamma$, so $\mcg(\Sigma)$ acts on all these vector spaces.
\begin{rmk}
	As already recalled in the introduction, the Dehn-Nielsen-Baer theorem gives a canonical isomorphism between $\mcg(\Sigma)$ and a subgroup of index 2 of $\out(\Gamma)=\aut(\Gamma)/\inn(\Gamma)$ \cite{farb2011primer}.
	The natural action $\out(\Gamma)\acts H^2_b(\Gamma)$ restricts to $\mcg(\Sigma)<\out(\Gamma)$;
	an element of $\mcg(\Sigma)$ represented by a diffeomorphism $h$ (fixing the basepoint of $\Sigma$) acts on $H^2_b(\Gamma)$ as the induced automomorphism $h_*\colon \Gamma\to\Gamma$.
\end{rmk}

We now describe more directly the action on $ZL^\infty_{\rm alt} ((S^1)^3)^\Gamma$.
We fix, once and for all, a hyperbolic metric (denoted by $m_0$ in Subsection \ref{sec: bc as meas fun}) on $\Sigma$ and an isometry between its universal cover and $\Hp$.
Any diffeomorphism $h\colon\Sigma\to\Sigma$ representing an element $\varphi=[h] \in \mcg(\Sigma)$ can be lifted to a map $\tilde h \colon \Hp\rightarrow \Hp$.
Notice that the lift of the representative $h$ is not canonical, but it is uniquely determined \emph{up to post-composition with elements of $\Gamma$}.
Since $\tilde h$ is a quasi-isometry \cite[Section 2]{minski2013brief}, it extends to the boundary at infinity, and we call $\partial \tilde h$ the induced homeomorphism on $S^1 = \bi\Hp$.
Up to composing with an isometry in $\Gamma$, $\partial \tilde h$ does not depend on the chosen representative $h$ of the mapping class \cite{minski2013brief}.

\begin{rmk}
The space $S^1 \times S^1 \times S^1 \setminus \Delta$ can be identified with the disjoint union of two copies of the unit tangent bundle $T^1 \Hp$, and its quotient by $\Gamma$ with the disjoint union of two copies of $T^1 \Sigma$. The corresponding action of $\mcg(\Sigma)$ on this space has also been studied in \cite{Souto10}.
\end{rmk}

We want to give a convenient description of the action of $\mcg(\Sigma)$ on $ZL^\infty_{\rm alt} ((S^1)^3)^\Gamma$.
The candidate formula is the following: given $\phi\in \mcg(\Sigma)$, $f \in L^\infty_{\rm alt} ((S^1)^3)^\Gamma$ and $(\xi_0,\xi_1,\xi_2)\in (S^1)^3$, we set
\[
	(\phi \cdot f) (\xi_0,\xi_1,\xi_2) = f \left(\partial\tilde h^{-1}( \xi_0),\partial\tilde h^{-1}(\xi_1),\partial\tilde h^{-1}(\xi_2)\right),
\]
where $h$ is any diffeomorphism representing $\phi$ and $\tilde h$ is any of its lifts on $\Hp$.
However, \emph{this formula is not well defined}, since $\partial \tilde h$, albeit continuous, might not preserve the Lebesgue measure-class, i.e., images of subsets of zero Lebesgue measure might have positive Lebesgue measure.\footnote{We thank the referee for having pointed out this issue, which we had overlooked in the first version of the paper.}
With the following lemma we show however that the formula in fact holds for cocycles in $CH_b^2(\Gamma)$.
Recall that there is a natural $\Gamma$-equivariant chain map
\[
\psi\colon L^\infty_{\rm alt}((S^1)^3)\rightarrow C^2_b(\Gamma) ,
\]
defined in (\ref{eq: chain map psi}), which induces the isomorphism in (\ref{eq: bc is measurable functions}).
\begin{lemma}
	\label{lem: formula for action of mapping class group}
	Let $f\colon (S^1)^3\setminus\Delta\rightarrow \R$ be a map in $CH^2_b(\Gamma)$ and consider $\phi\in \mcg(\Sigma)$, $h$ any diffeomorphism representing $\phi$ and $\tilde h$ any of its lifts on $\Hp$.
	Then the function $g\colon (S^1)^3\setminus \Delta\rightarrow \R$ defined by the formula
	\[
	g(\xi_0,\xi_1,\xi_2)=f(\partial \tilde h^{-1}(\xi_0), \partial \tilde h^{-1}(\xi_1), \partial \tilde h^{-1}(\xi_2)) \,
	\]
	is an element in $CH_b^2(\Gamma)$ representing the class of $\varphi \cdot \psi(f)$ in $H^2_b(\Gamma)$.
	In particular, the subspace $\calC \subseteq H^2_b(\Gamma)$ is preserved by the action $\mcg(\Sigma)\acts H^2_b(\Gamma)$.
\end{lemma}
\begin{proof}
	Because of continuity, we have that $f$ is $\Gamma$-invariant and satisfies the cocycle condition $\delta f \equiv 0$ (everywhere, not just almost everywhere).
	We denote by $h_*$ the automorphism of $\Gamma$ induced by $h$.
	The following formula can be deduced from standard covering theory: for every $\gamma \in \Gamma$ and every $\xi \in S^1$
	\[
	h_*(\gamma)\cdot \partial \tilde h(\xi)=\partial \tilde h(\gamma \cdot \xi) .
	\]
	Thus, the map $g\colon (S^1)^3\setminus \Delta \rightarrow \R$ is a bounded  $\Gamma$-invariant continuous map representing an element in $L^{\infty}_{\rm alt} ((S^1)^3)^\Gamma$. Moreover, as the map sending $(\xi_0, \xi_1,\xi_2)$ to $(\partial \tilde h^{-1}(\xi_0), \partial \tilde h^{-1}(\xi_1), \partial \tilde h^{-1}(\xi_2))$ preserves $\Delta$, we have that $\delta g\equiv 0 $. In particular, $g$ represents an element in $CH^2_b(\Gamma)$.   

	It remains to show that $\phi \cdot \psi(f)$ and $\psi(g)$ are cohomologous in $C^\bullet_b(\Gamma)^\Gamma$. 
	We have that
	\begin{align*}
		& (\varphi\cdot \psi(f) - \psi(g))(\gamma_0,\gamma_1
		,\gamma_2)\\
		& = \psi(f)((h^{-1})_*(\gamma_0), (h^{-1})_*(\gamma_1), (h^{-1})_*(\gamma_2))- \psi(g)(\gamma_0,\gamma_1, \gamma_2)\\
		& = \int_{(S^1)^{3}} \Big(f \big((h^{-1})_*(\gamma_0), \dots \big) - g(\gamma_0\cdot x_0,\dots)\Big) \ d\mu \\
		& =  \int_{(S^1)^{3}}\Big( f \big(\partial \tilde h^{-1}(\gamma_0\cdot \partial \tilde h (x_0)), \dots \big) - g(\gamma_0\cdot x_0,\dots)\Big) \ d\mu .
	\end{align*}
	If we set 
	\begin{align*}
	\beta(\gamma_0,\gamma_1) =\int_{(S^1)^{3}} \Big(&g \big(\gamma_0\cdot \partial \widetilde h (x_0), \gamma_1\cdot \partial \widetilde h (x_1), \gamma_1\cdot x_1 \big) \\&- g \big(\gamma_0\cdot \partial \widetilde h (x_0), \gamma_0\cdot x_0, \gamma_1\cdot x_1 \big)\Big) \ d\mu  ,
\end{align*}
	then it follows from $\delta g = 0$ that $\beta \in C^1_b(\Gamma)$ is a $\Gamma$-equivariant bounded cochain such that $\varphi\cdot \psi(f) - \psi(g) = \delta \beta$.
\end{proof}
\begin{rmk}
	The previous lemma can also be derived as a consequence of \cite[Proposition 1, Corollary 2]{burger2002boundary},  where a chain complex of Borel regular functions on the boundary (not function classes but actual functions) is considered.
	In the same way, we could have streamlined part of the proof of Lemma \ref{lem: indipendence of calC from metric}.
	We thank the referee for this observation.
\end{rmk}
Since we are considering orientation-preserving maps, it clearly follows from Lemma \ref{lem: formula for action of mapping class group} that the Euler class is a fixed point of the action $\mcg(\Sigma) \acts \calC$.
The main result of this work states that the multiples of the Euler class are, in fact, the only fixed classes of this action; even more, the subspace generated by the Euler class is the only finite-dimensional nontrivial $\mcg(\Sigma)$-invariant subspace.
\begin{varthm}{thm: invariant continuous subspacesintro}
	\label{thm: invariant continuous subspaces}
	The subspace  generated by the Euler class is the only finite-dimensional nontrivial subspace preserved by the action $\mcg(\Sigma) \acts \calC$.
\end{varthm}
We devote the next sections to the proof of this result.

\subsection{Action of a Dehn twist}
Let $\sigma$ be a simple closed geodesic in $\Sigma$. With a sligth abuse of notation (but for a better readibility), we denote by $\tau_{\sigma}\in \mcg(\Sigma)$ both a Dehn twist in $\Sigma$ along the geodesic $\sigma$ and the corresponding mapping class in $\mcg(\Sigma)$ (specifying what we mean when this notation might be ambigous).

To understand the behaviour at infinity of a lift of a $\tau_{\sigma}$ (as a homeomorphism of $S^1 = \bi\Hp$), it is enough to understand its \emph{coarse} action on  $\Hp$.
In order to do that, we describe some (discontinuous) maps from the hyperbolic plane to itself, which are called \emph{earthquakes} along $\sigma$ (see, e.g., \cite[Section 2]{minski2013brief}, where some wonderful pictures are displayed).
Every earthquake along $\sigma$ will be close to a lift of $\tau_ \sigma$; hence, it will induce on $\bi\Hp$ the same homeomorphism induced by that lift.

Consider the preimage in $\Hp$ of the geodesic $\sigma \subset \Sigma$.
It consists of countably many disjoint lines, dividing $\Hp$ into many pieces that are called \emph{plates}.
Choose one of these plates, and call it the \emph{central plate}; we denote it by $\mathcal{P}$.
The boundary of $\mathcal{P}$ consists of infinitely many geodesic lines, corresponding to the components of the complement of $\mathcal{P}$.
We call these components \emph{affected regions}.
Every affected region is a halfplane containing infinitely many plates (see Figure \ref{fig: lifts of simple closed geod}).
\begin{figure}[ht]
	\includegraphics{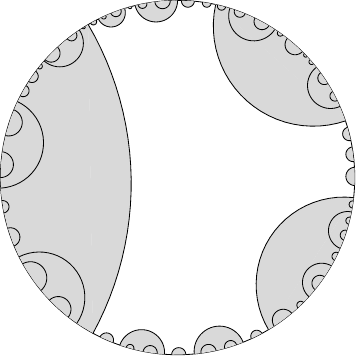}
	\caption{The white region represents the central plate; the grey regions are the affected regions.}
	\label{fig: lifts of simple closed geod}
\end{figure}

An earthquake $\eqs$ along $\sigma$ can be described as follows:
\begin{itemize}
	\item It is the identity on the central plate $\mathcal{P}$;
	\item If $\mathcal{Q}$ is a plate adjacent to $\mathcal{P}$ along a geodesic $\tilde\sigma$ (a lift of $\sigma$), then the earthquake acts on $\mathcal{Q}$ as the orientation-preserving hyperbolic isometry with axis $\tilde\sigma$ and translation length equal to the length of $\sigma$, in the direction that moves $\tilde\sigma$ counter-clockwise around $\mathcal{P}$;
	\item It is extended $\Gamma$-equivariantly on the other plates.
\end{itemize}
We remark once again that $\eqs$ stays at bounded distance from a lift of the map $\tau_\sigma$. However, it cannot be a lift of $\tau_\sigma$ itself, as it is discontinuous.

The boundaries at infinity of $\mathcal P$ and of an affected region $\mathcal R$ are denoted by $\bi \mathcal P$ and $\bi \mathcal R$ respectively.
Both these sets are closed subsets of $S^1$.
In particular, $\bi \mathcal R$ is a closed interval, while $\bi \mathcal P$ is a Cantor set.
Note that the union of the boundaries at infinity of the affected regions is dense in $S^1$.
Moreover, whenever $\mathcal R_1$ and $\mathcal R_2$ are two distinct affected regions, then $\bi \mathcal R_1 \cap \bi\mathcal R_2$ is empty.

The dynamics of $\partial\eqs:S^1\to S^1$ can be described in the following way.
The set of fixed points of $\partial\eqs$ is given by $\bi \mathcal{P}$.
For every affected region $\mathcal R$, the closed interval $\bi \mathcal R$ is preserved by $\partial\eqs$, but not pointwise.
The endpoints of $\bi \mathcal R$, that are also contained in $\bi \mathcal{P}$, are the only fixed points in the interval: one of them is attractive, while the other is repulsive.
In particular, if $\xi$ is any point contained in the interior of $\bi\mathcal R$, then $(\partial\eqs)^k(\xi)$ converges to the attractive fixed point as $k$ goes to $+\infty$, while it converges to the repulsive fixed point as $k$ goes to $-\infty$.

\begin{rmk}
	\label{rmk: moving central plates by conjugation}
	Different choices of $\mathcal{P}$ yield different earthquakes along $\sigma$, and any two of these earthquakes differ by the composition with an isometry in $\Gamma$.
	We point out (even though this does not play any role in our arguments) that not every composition of an earthquake with an isometry in $\Gamma$ is an earthquake along $\sigma$, because for it to be an earthquake it must fix pointwise one of the plates, which is not guaranteed in general.
	Accordingly, not every lift of a Dehn twist is close to an earthquake.
\end{rmk}

\begin{defn}\label{3regprop}
	We say that $f \in \calC$ has the \emph{3-region property} if, for every simple closed geodesic $\sigma$ of $\Sigma$, for every earthquake $\eqs$ along $\sigma$, and for every triple $\mathcal R_1,\, \mathcal R_2,\,\mathcal R_3$ of distinct affected regions (with respect to the central plate of $\eqs$), we have that $f$ is constant on $\bi\mathcal R_1 \times \bi\mathcal R_2 \times \bi\mathcal R_3 \subseteq (S^1)^3$.
\end{defn}

\begin{lemma}\label{lemma: constant on affected region}
	If $f\in \calC$ belongs to a finite-dimensional subspace $V$ of $\calC$ that is invariant under the action of $\mcg(\Sigma)$, then $f$ has the 3-region property.
\end{lemma}
\begin{proof}
	We consider two norms on $V$:
	\begin{itemize}
		\item The sup-norm, denoted by $\lVert \cdot\rVert_\infty$;
		\item The $\ell^1$-norm with respect to a fixed basis $\mathcal{B}=\{f_1, \ldots, f_k\}$ of $V$, defined as
		      \[\lVert \beta_1 f_1 + \ldots + \beta_k f_k \rVert_{1, \mathcal{B}}=\sum_{i=1}^k |\beta_i|.\]
	\end{itemize}

	Fix a simple closed geodesic $\sigma$ of $\Sigma$, a Dehn twist $\tau_\sigma$ along $\sigma$, an earthquake $\eqs$ along $\sigma$, and a triple $\mathcal R_1,\, \mathcal R_2,\,\mathcal R_3$ of distinct affected regions.

	Let $f \in V$.
	Since $V$ is $\tau_\sigma$-invariant, for every $n \in \Z$ there are real numbers $\alpha^{(n)}_{1}, \ldots, \alpha_{k}^{(n)}$ such that
	\[ \tau_\sigma^{n}\cdot f= \alpha^{(n)}_{1} f_1 + \ldots + \alpha^{(n)}_{k} f_k.\]
	By the description of the action given in Lemma \ref{lem: formula for action of mapping class group}, we also have that $\lVert f \rVert_\infty=\lVert \tau^{n}_\sigma\cdot f \rVert_\infty$.
	Since all norms on a finite-dimensional vector space are equivalent, there exists $K>0$ such that $\lVert \tau^{n}_\sigma \cdot f \rVert_{1, \mathcal{B}} \leq K \cdot \lVert \tau^n_\sigma\cdot f \rVert_\infty = K \cdot  \lVert f \rVert_\infty$.
	Therefore, the constant $M_f = K \cdot \|f\|_\infty>0$ (depending only on $f$) satisfies
	$\left|\alpha_{i}^{(n)}\right| \leq M_f$,
	for every $i\in \{1, \ldots, k\}$ and $n \in \Z$.

	For $j\in \{1,2,3\}$, let $a_j^+ \in \bi \mathcal R_j$ be the attractive point of $\bi \mathcal{R}_j$ and let $\xi_j$ be any point in the interior of $\bi \mathcal R_j$.
	Since the affected regions $\mathcal R_j$ are distinct, then also the points $a^+_j$ are pairwise distinct.
	Therefore the triple $(a^+_1,a^+_2,a^+_3)\in (S^1)^3$ does not lie in the multidiagonal $\Delta$.
	By Lemma \ref{lem: formula for action of mapping class group}, for every $n \in \mathbb{Z}$ we have that
	\[f(\xi_1, \xi_2, \xi_3)=(\tau_\sigma^{-n}\cdot f)(\partial\eqs^n (\xi_1), \partial\eqs^n (\xi_2), \partial\eqs^n (\xi_3)).
	\]
	Since $f_i  \in \mathcal  B \subset \calC$ is continuous on $(S^1)^3\setminus \Delta$, for every $\varepsilon>0$ there exists $n_0\in \mathbb{Z}$ such that, for each $i\in \{1, \dots, k\}$ and for each $n\in \Z_{\geq n_0}$ we have that
	\[
		\left|f_i \left(\partial\eqs^n(\xi_1),\partial\eqs^n(\xi_2),\partial\eqs^n(\xi_3)\right)
		-f_i \left(a^+_1,a^+_2,a^+_3 \right)\right|
		\leq \frac{\varepsilon}{k \cdot M_f}.
	\]
	Using the fact that $a^+_1, a^+_2$ and $a^+_3$ are fixed by $\partial\eqs^n$, it follows that
	\begin{align*}
		     & \left|f(\xi_1, \xi_2, \xi_3)-f \left(a^+_1,a^+_2,a^+_3\right)\right|                                                                                                                      \\
		=    & \left|(\tau_{\sigma}^{-n}\cdot f)\left(\partial\eqs^n(\xi_1),\partial\eqs^n(\xi_2),\partial\eqs^n(\xi_3)\right)-(\tau_{\sigma}^{-n}\cdot f) \left(a^+_1,a^+_2,a^+_3\right)\right|         \\
		=    & \left\lvert \sum_{i=1}^k \alpha_i^{(-n)} \left( f_i \left(\partial\eqs^n(\xi_1),\partial\eqs^n(\xi_2),\partial\eqs^n(\xi_3)\right)- f_i\left(a^+_1,a^+_2,a^+_3\right)\right) \right\rvert \\
		\leq & M_f\cdot \sum_{i=1}^k \left|f_i\left(\partial\eqs^n(\xi_1),\partial\eqs^n(\xi_2),\partial\eqs^n(\xi_3)\right)-f_i(a^+_1,a^+_2,a^+_3)\right|                                               \\
		\leq & M_f\cdot \frac{\varepsilon}{M_f}=\varepsilon.
	\end{align*}
	Since $\varepsilon$ was arbitrary, we obtain that
	$f(\xi_1, \xi_2, \xi_3) =f\left(a^+_1,a^+_2,a^+_3\right)$,
	and therefore $f$ is constant on the interior of $\bi\mathcal R_1 \times \bi\mathcal R_2 \times \bi\mathcal R_3 \subseteq (S^1)^3$.
	By continuity, $f$ is constant on $\bi\mathcal R_1 \times \bi\mathcal R_2 \times \bi\mathcal R_3$.
\end{proof}

\section{Proof of Theorem \ref{thm: invariant continuous subspaces}}\label{Sec:proof}
In this section we prove that if $f \in \calC$ has the 3-region property, then it is locally constant (Lemma \ref{lem: continuous invariants are locally constant}), and then deduce Theorem \ref{thm: invariant continuous subspaces}.
The strategy is to ``separate'' triples of distinct points in $\bi \Hp$, trying to put them in distinct affected regions with respect to some earthquake, so that the 3-region property can be applied.
But we cannot accomplish this for every triple; once a triple of points is given, the best we can do in general is to find an earthquake so that we fall in one of the following three situations:
\begin{itemize}
	\item The three points belong to three distinct affected regions;
	\item Two of the points belong to two distinct affected regions, and the remaining one is in the boundary at infinity of the central plate;
	\item One of the points belongs to an affected region, while the other two are in the boundary at infinity of the central plate.
\end{itemize}
In the first case, the 3-region property readily implies that $f$ is constant in a neighbourhood of the triple.
In the other two cases, we have to resort to more involved arguments (Lemmas \ref{lemma: difficile} and \ref{lem: continuous invariants are locally constant}).

We start by proving the following general result, stating that for every finite subset $\mathcal G$ of closed geodesics in a compact Riemannian manifold $M$, every geodesic ray $\gamma \colon [0,\infty)\to M$ is either definitely close to an element of $\mathcal G$ or far apart from every element of $\mathcal G$ for arbitrarily large times $t \in[0,\infty)$.

\begin{lemma}\label{lemma:fellow_geodesic}
	Let $M$ be a closed connected Riemannian manifold, and let $\mathcal{G} = \{\sigma_1, \dots, \sigma_k\}$ be a finite set of closed geodesics in $M$.
	Then there is a constant $\varepsilon > 0$ (which can be taken as small as desired) such that, for every geodesic ray $\gamma:[0,+\infty)\to M$, one of the following holds:
	\begin{enumerate}
		\item There are arbitrarily large times $t \in [0,+\infty)$ such that the point $\gamma(t)$ has distance at least $\varepsilon$ to each of the geodesics in $\mathcal{G}$;
		\item There are a closed geodesic $\sigma_i \in \mathcal{G}$ and a time $t_0$ such that, for every $t \ge t_0$, the point $\gamma(t)$ has distance strictly smaller than $\varepsilon$ from $\sigma_i$.
	\end{enumerate}
\end{lemma}
\begin{proof}
	To begin with, notice that if the geodesics in $\mathcal{G}$ are disjoint then the conclusion follows by taking $\varepsilon$ smaller than half the smallest distance between two geodesics in $\mathcal{G}$. This is indeed positive as the geodesics are closed.
	We now consider the case in which there are intersections.

	Let $\mathcal I \subset M$ be the set of intersections between (at least two) geodesics of $\mathcal G$; it is a finite set of points.
	For every  $p \in \mathcal I$ fix a small closed convex neighbourhood $N_p$, in such a way that $N_p \cap N_q =\emptyset$ whenever $p \neq q$.
	Define $T_p$ to be the maximal length of a geodesic segment contained in $N_p$, and for every pair $(p,q)\in \mathcal I \times \mathcal I$, define $T_{pq}$ as follows:
	\begin{itemize}
		\item If $p \neq q$, then $T_{pq}$ is the distance between $N_p$ and $N_q$;
		\item If $p = q$, then $T_{pp}$ is the infimum of the lengths of geodesic (\emph{locally} distance-minimizing) segments with endpoints in $N_p$ and not completely contained in $N_p$. Let $T_{pp}=\infty$ if there is no such geodesic.
	\end{itemize}

	Let $T_1 = \max\{T_p : p\in\mathcal{I}\}$ and $T_2 = \min\{T_{pq}: (p,q)\in\mathcal{I}\times\mathcal{I}\}$.
	If the neighbourhoods have been chosen small enough (as we assume), then $T_1 < T_2$.
	We denote by $V$ the (disjoint) union of the neighbourhoods $N_p$.
	Let $\varepsilon_0 > 0$ be such that, for every $\sigma_i \in \mathcal G$ and for every point lying on $\sigma_i$, but not belonging to $V$, its distance to geodesics in $\mathcal{G}\setminus \{\sigma_i \}$ is greater than $2\varepsilon_0$.

	We prove the statement by contradiction, assuming that there is a sequence $\{\varepsilon_n\}_{n\in \N_{\ge 1}}$ of positive numbers converging to $0$ and a sequence of unit-speed geodesic rays $\gamma_n$ such that:
	\begin{enumerate}
		\item\label{item: first} The ray $\gamma_n$ is definitively $\varepsilon_n$-close to the union $\sigma_1 \cup \dots \cup \sigma_k$;
		\item\label{item: second} For every $\sigma_i \in \mathcal{G}$ and every $n\in\N_{\ge 1}$, the ray $\gamma_n$ is outside the $\varepsilon_n$-neighbourhood of $\sigma_i$ for arbitrarily large times.
	\end{enumerate}
	We can assume that $\varepsilon_n < \varepsilon_0$ for every $n \in \N_{\ge 1}$.
	Putting together these two conditions with the definition of $\varepsilon_0$, one gets that the geodesic ray $\gamma_n$ must enter $V$ infinitely many times: indeed $\gamma_n$ is (definitely) always close to at least one geodesic of $\mathcal G$ (\ref{item: first}), but not always the same geodesic (\ref{item: second}); this means that there are infinitely many times in which $\gamma_n$ is close to at least two geodesics of $\mathcal G$ (this certainly happens when $\gamma_n$ switches the geodesic to which it is close).
	More precisely, the set
	\[\{t \in [0,+\infty): \gamma_n(t) \in V\}\]
	is an infinite disjoint union of closed intervals of length at most $T_1$, and any two of these intervals are at distance at least $T_2$ apart.

	Let $T$ be a real number such that $T_1 < T < T_2$.
	The discussion above implies the existence of $a_n \in [0,+\infty)$ such that the two points $\gamma_n(a_n), \gamma_n(a_n+T)$ do not belong to $V$ and are $\varepsilon_n$-close to two distinct geodesics of $\mathcal{G}$.
	Up to precomposing the rays with translations of $[0,+\infty)$, we can assume that $a_n = 0$ for every $n \in \N_{\ge 1}$ (i.e., we redefine the domain of $\gamma_n$ in such a way that this ray starts at $\gamma_n(a_n)$).

	By passing to a subsequence if necessary, we can assume (by the Ascoli-Arzelà Theorem) that the sequence of rays $\gamma_n$ converges pointwise to a geodesic ray $\gamma$.
	Since $\varepsilon_n$ goes to $0$, this limit ray must be contained in $\sigma_1 \cup \dots \cup \sigma_k$ but, being a geodesic itself, it has to be contained in some $\sigma_i$.
	On the other hand, $\gamma(0) = \lim_{n\to\infty}\gamma_n(0)$ and $\gamma(T) = \lim_{n\to\infty}\gamma_n(T)$ cannot belong to the same $\sigma_i$, otherwise $\gamma_n(0)$ and $\gamma_n(T)$ would be both $\varepsilon_0$-close to $\sigma_i$, for $n$ big enough.
	This provides a contradiction.
\end{proof}

\begin{lemma}
	\label{lem: nice_cut_geodesic}
	Let $\xi_1, \xi_2$ and $\xi_3$ be three pairwise distinct points of $\bi \Hp$.
	Then, there exists a simple closed geodesic in $\Sigma$ such that one of its lifts in $\Hp$ separates $\xi_1$ and $\xi_2$ from $\xi_3$, meaning that the endpoints of this lift divide $\bi\Hp$ into two open arcs, one containing $\xi_3$ and the other containing $\xi_1$ and $\xi_2$.
	Moreover, $\sigma$ can be chosen in such a way that none of its lifts in $\Hp$ is asymptotic to $\xi_1$, $\xi_2$, or $\xi_3$.
\end{lemma}
\begin{proof}
	Let $\mathcal{G} = \{\sigma_1, \dots, \sigma_k\}$ be a finite set of simple closed geodesics in $\Sigma$ cutting the surface into polygons (e.g., hexagons).
	We denote by $\tilde{\mathcal{G}}$ the (infinite) set of geodesic lines in $\Hp$ that project onto geodesics in $\mathcal{G}$.
	Notice that a point in $\bi \Hp$ can serve as the endpoint for at most one geodesic projecting to a simple closed geodesic on the surface $\Sigma$. Consequently, there can be a maximum of three simple closed geodesics admitting a lift asymptotic to $\xi_1,\xi_2$, or $\xi_3$. In particular, we can assume, up to changing the choice of $\mathcal{G}$, that none of the lines in $\tilde{\mathcal{G}}$ is asymptotic to $\xi_1, \xi_2$ and $\xi_3$.

	Let $\varepsilon > 0$ be such that the conclusion of Lemma \ref{lemma:fellow_geodesic} holds and suppose it is small enough so that pairs of distinct geodesics in $\tilde{\mathcal{G}}$ projecting onto the same geodesic in $\mathcal{G}$ are at distance at least $2\varepsilon$ apart.

	The following construction is described in Figure \ref{fig: lemma42}.
	Let $\tilde\gamma:[0,+\infty) \to \Hp$ be a geodesic ray asymptotic to $\xi_3$.
	For $i \in \{1,2\}$, we denote by $\alpha_i(t) \in [0,\pi]$ the angle between the rays that, starting at $\tilde\gamma(t)$, go towards $\xi_i$ and $\xi_3$.
	We have that $\alpha_1(t)\to \pi$ and $\alpha_2(t) \to \pi$ as $t \to +\infty$.

	There are arbitrarily large times $t\in[0,+\infty)$ such that $\tilde\gamma(t)$ has distance at least $\varepsilon$ from the union of the lines in $\tilde{\mathcal{G}}$:
	indeed, if this were not the case, then the projection of $\tilde\gamma$ on $\Sigma$ would stay definitively $\varepsilon$-close to $\mathcal{G}$, and by Lemma \ref{lemma:fellow_geodesic} it would stay definitively $\varepsilon$-close to some specific $\sigma_i \in \mathcal{G}$.
	But this would imply that $\tilde\gamma$ stays $\varepsilon$-close to some lift of $\sigma_i$ (recall that two different lifts are at distance at least $2\varepsilon$ apart), and this lift would then converge to $\xi_3$, a situation that we excluded at the beginning of the proof.

	Let $t \in [0,+\infty)$ be such that the point $\tilde\gamma(t)$ has distance at least $\varepsilon$ from the union of the lines in $\tilde{\mathcal{G}}$.
	We claim that there exists a $\delta>0$, depending on $\varepsilon$, that satisfies the following property:
	for every geodesic $\tilde \sigma$ in $\tilde{\mathcal G}$, the angle $\beta$ centered in $\tilde\gamma(t)$ between the extrema $\xi^+,\ \xi^-\in\bi\Hp$ of $\tilde\sigma$, is smaller than $\pi-\delta$.
	Indeed, the geodesic ideal triangle $T$ with vertices $\xi^+,\ \xi^-,\ \gamma(t)$ satisfies
	\[\pi -\beta = {\rm{Area}} (T) >{\rm{Area}} (T \cap B_{\varepsilon} (\gamma(t))) =\beta \cdot \frac{{\rm{Area}} (B_\varepsilon(\gamma(t)))} {2 \pi},\]
	where the first equality holds because two of the angles of $T$ are $0$ and the last one is $\beta$. This inequality proves the claim.

	If $t$ is big enough, we can assume that $\alpha_1(t) > \pi-\delta$ and $\alpha_2(t) > \pi-\delta$.
	The point $\tilde\gamma(t)$ lies inside a convex polygon whose sides are segments of lines in $\tilde{\mathcal{G}}$.
	If $\ell \in \tilde{\mathcal{G}}$ is one of these lines, its endpoints divide $\bi\Hp$ into two connected components: a ``big'' one, corresponding to the half-plane containing $\tilde\gamma(t)$, and a ``small'' one, corresponding to the other half-plane into which $\Hp$ is cut by $\ell$.
	The small one, seen from $\tilde\gamma(t)$, is less than $\pi-\delta$ wide.

	Among the lines forming the polygon, there must be one such that the corresponding ``small'' component of $\bi\Hp$ contains $\xi_3$.
	This component cannot contain $\xi_1$ or $\xi_2$, since from the point of view of $\tilde\gamma(t)$ they are at distance greater than $\pi-\delta$ from $\xi_3$, and, on the other hand, the component is less than $\pi-\delta$ wide. Thus, this line is a lift of a geodesic in $\mathcal G$ that satisfies all the requirements of the statement.
\end{proof}

\begin{figure}[ht]
	\includegraphics[scale=1.2]{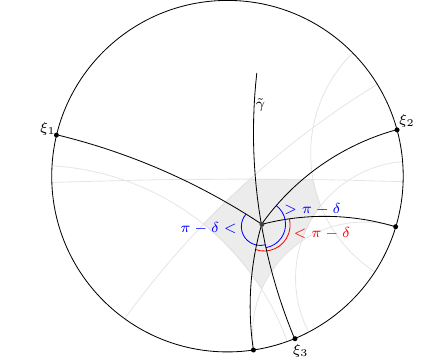}
	\caption{Cartoon of the proof of Lemma \ref{lem: nice_cut_geodesic}. The grey geodesics represents some geodesics in $\tilde{\mathcal G}$; the area in grey represents one lift of one polygon in which the surface is cut; the two blue angles represents $\alpha_1(t)$ and $\alpha_2(t)$; the red angle is the visual angle $\beta$ that $\gamma(t)$ sees when looking to one of the grey geodesic in $\tilde{\mathcal G}$. }
	\label{fig: lemma42}
\end{figure}
\begin{lemma}\label{lemma: geodetica separante}
	Let $\xi_1,\xi_2,\xi_3 \in S^1$ be three distinct points. Then there is an earthquake $\eqs$ along a simple closed geodesic $\sigma \subset \Sigma$ such that:
	\begin{enumerate}
		\item\label{item: separating affected region} There is an affected region $\mathcal T$ of $\eqs$ such that $\xi_3$ is contained in the interior of $\bi \mathcal T$ and $\xi_1,\xi_2 \notin \bi \mathcal T$;
		\item\label{item: additional condition} The points $\xi_1$ and $\xi_2$ do not belong to the boundary at infinity of the same affected region of $\eqs$.
	\end{enumerate}
\end{lemma}
Note that $\xi_1$ and $\xi_2$ might either belong to the boundaries at infinity of two different affected regions, or (one of them or both) to the boundary at infinity of the central plate.
\begin{proof}
	By Lemma \ref{lem: nice_cut_geodesic}, there exist a simple closed geodesic $\sigma \subset \Sigma$ and an earthquake $\eqs'$ along $\sigma$ satisfying Condition (\ref{item: separating affected region}). Furthermore, we can assume that no lift of $\sigma$ in $\Hp$ is asymptotic to $\xi_1$, $\xi_2$, or $\xi_3$.
	We prove that, changing suitably the central plate, there is a (possibly different) earthquake along $\sigma$ satisfying also Condition (\ref{item: additional condition}).

	Let $\alpha$ be the geodesic line with endpoints $\xi_1$ and $\xi_2$.
	Let $s\colon [0,1]\to\Hp$ be a geodesic segment with $s(0)\in \mathcal T$ and $s(1)\in \alpha$.
	By compactness, the path $s$ intersects a finite number of lifts of $\sigma$. Let $\overline{t } \in[0,1]$ be the largest time $t$ for which the following happens: the point $s(t)$ belongs to some lift of $\sigma$ whose points at infinity separate $\xi_3$ from $\xi_1$ and $\xi_2$, as in Lemma \ref{lem: nice_cut_geodesic} (see Figure \ref{fig: trova il dt appropriato}).
	Note that this happens at least once (when $s$ exits $\mathcal{T}$).
	Call $\tilde{\sigma}$ the lift of $\sigma$ that intersects $s$ in $s(\overline{t })$.
	\begin{figure}[ht]
		\includegraphics[scale=1.3]{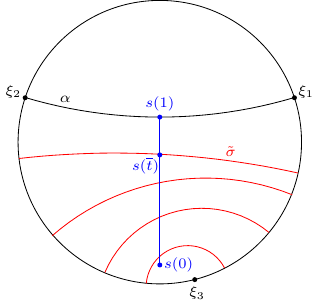}
		\caption{The geodesic $\tilde \sigma$ is the last lift of $\sigma$ intersecting the path $s$ and separating $\xi_3$ from $\xi_1$ and $\xi_2$}
		\label{fig: trova il dt appropriato}
	\end{figure}

	Consider the earthquake $\eqs$ along $\sigma$ whose central plate contains $s(\overline{t }+\varepsilon)$ for small values of $\varepsilon>0$.
	Now, $\eqs$ satisfies also Condition (\ref{item: additional condition}). Indeed, if $\xi_1$ and $\xi_2$ were still contained in the boundary at infinity of the same affected region $\mathcal R$ of $\eqs$, then the geodesic bounding $\mathcal R$ would be a lift of $\sigma$ lying between $\alpha$ and $\tilde \sigma$, thus it would intersect $s$ in $s(t')$, with $t'>\overline{t}$, and we would get a contradiction.
\end{proof}

\begin{lemma}\label{lemma: difficile}
	Let $\mathcal R_1$ and $\mathcal R_2$ be distinct affected regions of an earthquake $\eqs$ along a closed geodesic $\sigma\subset\Sigma$.
	Let $I$ be one of the two connected components of $S^1 \setminus (\bi\mathcal R_1 \cup \bi\mathcal R_2)$.
	If $f \in \calC$ has the 3-region property, then $f$ is constant on $\bi R_1\times\bi  \mathcal R_2 \times I\subseteq (S^1)^3$.
\end{lemma}
\begin{proof}
	Let $\mathcal{R}_3$ be an affected region of $\eqs$ whose boundary at infinity is contained in $I$, and for $i \in \{1,2,3\}$ fix a point $\xi_i^o$ in the interior of $\bi\mathcal{R}_i$.
	The set
	\[\mathcal S = \{\xi_3 \in I\,|\,f(\xi_1,\xi_2,\xi_3)= f(\xi_1^o,\xi_2^o,\xi_3^o)\quad \forall (\xi_1,\xi_2) \in \bi\mathcal{R}_1 \times \bi\mathcal{R}_2\}\,\]
	is nonempty (by the 3-region property) and closed in $I$, since the map $f$ is continuous on $\bi \mathcal R_1 \times\bi  \mathcal R_2 \times I$.
	We want to show that, under our assumptions, $\mathcal S$ is also open: this would imply $\mathcal{S} = I$, proving the lemma.

	Let $\xi_3 \in \mathcal S$.
	We want to find a neighbourhood $U \subseteq I$ of $\xi_3$ contained in $\mathcal S$.
	By applying Lemma \ref{lemma: geodetica separante} to the triple $(\xi_1^o,\xi_2^o,\xi_3)$, we find an earthquake $\eqsb$ relative to a simple closed geodesic
	in $\Sigma$ such that the following conditions hold:
	\begin{itemize}
		\item There is an affected region $\mathcal T_3$ of  $\eqsb$ such that $\xi_3$ is contained in the interior of $\bi \mathcal T_3$ and $\xi_1^o,\xi_2^o \notin \bi\mathcal T_3$;
		\item $\xi_1^o$ and $\xi_2^o$ are not contained in the boundary at infinity of the same affected region of $\eqsb$.
	\end{itemize}
	Since the union of the boundaries of affected regions of a given earthquake is dense in $\bi \Hp$, there are two distinct affected regions $\mathcal T_1$ and $\mathcal T_2$ of $\eqsb$ such that $\bi \mathcal T_i \cap \bi \mathcal R_i$ is nonempty for $i \in \{1,2 \}$ (see Figure \ref{fig: lemma-difficile}).
	Let us set $U = \bi \mathcal T_3 \cap I$, take an affected region $\mathcal{R}$ of $\eqs$ whose boundary at infinity intersects $U$, and consider $\xi \in U \cap \bi\mathcal{R}$.
	\begin{figure}[ht]
		\includegraphics[scale=1.3]{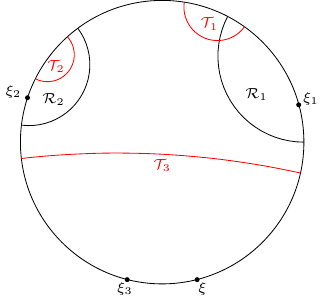}
		\caption{The affected regions $\mathcal T_1, \mathcal T_2$ and $\mathcal T_3$ of $\eqsb$ are distinct. Moreover, for $i \in \{1,2\}$, the boundary  at infinity of $\mathcal T_i$ has nonempty intersection with $\bi\mathcal R_i$.}
		\label{fig: lemma-difficile}
	\end{figure}
	Since $\mathcal R_1, \mathcal R_2$ and $\mathcal R$ are distinct and $f$ has the 3-region property, we get, for every
	$(\xi_1,\xi_2) \in \bi\mathcal{R}_1 \times \bi\mathcal{R}_2$,
	\begin{align*}
		f(\xi_1,\xi_2,\xi)\equiv &
		\restr{f} {\bi\mathcal R_1 \times \bi\mathcal  R_2 \times \bi\mathcal R}                                                                                    \\
		\equiv                   & \restr{f} {\bi(\mathcal R_1 \cap\mathcal T_1) \times \bi(\mathcal  R_2 \cap \mathcal T_2)\times \bi(\mathcal R\cap\mathcal T_3)} \\
		\equiv                   & \restr{f} {\bi\mathcal T_1 \times \bi\mathcal  T_2 \times \bi\mathcal T_3}                                                       \\
		\equiv                   & f(\xi_1',\xi_2',\xi_3)= f(\xi_1^o,\xi_2^o,\xi_3^o),
	\end{align*}
	where the symbol $\equiv$ means that the function is constant and equal to the other side.
	In the last line, $(\xi_1',\xi_2')$ is any pair in $\bi(\mathcal{R}_1\cap\mathcal{T}_1) \times \bi(\mathcal{R}_2\cap\mathcal{T}_2)$.
	Using again the fact that the union of the boundaries at infinity of affected regions of a given earthquake is dense in $\bi \Hp$ (and so in the open subset $U$), the continuity of $f$ ensures that $f(\xi_1,\xi_2,\xi')=f(\xi_1^o,\xi_2^o,\xi_3^o)$ holds for every point $\xi' \in U$.
	Therefore, $U \subset \mathcal S$.
\end{proof}
\begin{lemma}
	\label{lem: continuous invariants are locally constant}
	Let $f \in \calC$. If $f$ has the 3-region property, then it is locally constant on $(S^1)^3\setminus \Delta$.
\end{lemma}
\begin{proof}
	Let $(\xi_1,\xi_2,\xi_3)\in (S^1)^3\setminus \Delta$ and let $\eqs$ be an earthquake along $\sigma$ provided by Lemma \ref{lemma: geodetica separante}.
	Let us call $\mathcal R_3$ the affected region of $\eqs$ such that $\xi_3$ is contained in the interior of $\bi \mathcal R_3$ and $\xi_1,\xi_2 \notin \bi \mathcal R_3$.
	We assume the triple $(\xi_1,\xi_2,\xi_3)$ to be ordered counter-clockwise, the other case being completely analogous.
	Since boundaries at infinity of affected regions of $\eqs$ are dense in $S^1$, there exist two affected regions $\mathcal R_1, \mathcal R_2$ of $\eqs$ such that $\xi_1,\,\bi \mathcal R_1,\,\bi\mathcal R_2,\,\xi_2,\, \bi\mathcal R_3$ are in counter-clockwise order on $S^1$.
	Let $U_1$ be the connected component of $S^1 \setminus (\bi\mathcal R_1\cup \bi\mathcal R_3)$ containing $\xi_1$, and similarly let $U_2$ be the connected component of $S^1 \setminus (\bi\mathcal R_2\cup \bi\mathcal R_3)$ containing $\xi_2$.
	If $\mathcal R_1'$ and $\mathcal{R}_2'$ are affected regions of $\eqs$ such that $\bi\mathcal R_i'\subseteq U_i$, we have that
	\begin{align*}
		\restr{f}{\bi\mathcal R_1'\times \bi \mathcal R_2'\times \bi\mathcal R_3}
		 & \equiv \restr{f}{\bi\mathcal R_1'\times \bi \mathcal R_2\times \bi\mathcal R_3} \\
		 & \equiv \restr{f}{\bi\mathcal R_1\times \bi \mathcal R_2\times \bi\mathcal R_3},
	\end{align*}
	where we used that $f$ has the 3-region property and applied Lemma \ref{lemma: difficile} multiple times.
	Using again the fact that the boundaries at infinity of affected regions are dense in $S^1$ and that $f$ is continuous on $(S^1)^3\setminus \Delta$, we obtain that $f$ is constant on $U_1\times U_2\times \mathcal R_3$, which is a neighbourhood of $(\xi_1,\xi_2,\xi_3)$.
	This proves that the map $f$ is locally constant on $(S^1)^3\setminus \Delta$.
\end{proof}

We can now conclude the proof of Theorem \ref{thm: invariant continuous subspaces}.
Assume that there is a finite-dimensional subspace $V$ of $\calC$ that is invariant under the action $\mcg(\Sigma)\acts \calC$ and let $f \in V$.
By Lemma \ref{lemma: constant on affected region}, $f$ has the 3-region property.
We know that $(S^1)^3\setminus \Delta$ has two connected components: one corresponding to positively oriented triples and one corresponding to negatively oriented triples.
By Lemma \ref{lem: continuous invariants are locally constant}, $f$ is constant on these connected components, and, since it is alternating, it takes opposite values on them.
Therefore $f$ is a multiple of the orientation cocycle, or equivalently, of the Euler class.

\section{Quasimorphisms and exact bounded cohomology}\label{sec:qm vs bc}

In this section we give some details which may be helpful to the reader to have a better understanding of the relationship between de Rham quasimorphisms and de Rham classes in bounded cohomology.
Finally, we show how Theorem \ref{thm: invariant_quasi_2} and Corollary \ref{cor: invariant continuous quasimorphismsintro} are consequences of Theorem \ref{thm: invariant continuous subspacesintro}.

Let $\Omega^k(\Sigma)$ be the space of $k$-forms on $\Sigma$, $k \in \{1,2\}$.
We denote by $EH^2_b(\Gamma)$ the exact second bounded cohomology of $\Gamma$ in degree $2$, i.e., the kernel of the comparison map $\mathrm{comp}^2_\Gamma: H^2_b(\Gamma) \rightarrow H^2(\Gamma)$.
It is well known that $EH^2_b(\Gamma)$ is isomorphic to the space of homogeneous quasimorphisms $Q^h(\Gamma)$ on $\Gamma$ modulo homomorphisms \cite[Corollary 2.11]{frigerio2017bounded}:
\[
	EH^2_b(\Gamma)\cong Q^h(\Gamma)/\mathrm{Hom}(\Gamma, \R) .
\]
Moreover, since $\Sigma$ is (a model of) the classifying space of $\Gamma$, there is a natural isomorphism $H^2_b(\Sigma) \cong H^2_b(\Gamma)$.

We delve now in our specific framework. Fix a hyperbolic metric $m$ on $\Sigma$.
Recall from the introduction that there is a map $\Omega^1(\Sigma) \rightarrow Q^h(\Gamma)$, associating to each $1$-form $\alpha \in \Omega^1(\Sigma)$ the homogeneous quasimorphism $q^m_\alpha$ defined, for every $\gamma \in \Gamma$, as
\[q^m_\alpha(\gamma)= \int_{\rho^m_\gamma}\alpha , \]
where $\rho_\gamma^m$ is the closed geodesic in the free homotopy class of $\gamma$ according to the metric $m$.

The following diagram is commutative:
\begin{center}
	\begin{tikzcd}
		\Omega^1(\Sigma) \arrow[r] \arrow[dd, "d"] & Q^h(\Gamma) \arrow[d] \\
		& EH^2_b(\Gamma) \arrow[d,hook] \\
		\Omega^2(\Sigma)\arrow[r, "\Psi_m"] & H^2_b(\Gamma) ,
	\end{tikzcd}
\end{center}
where the left map is the exterior derivative, while $\Psi_m$ is defined in Subsection \ref{sub:derham clas}.
It follows that de Rham quasimorphisms in $Q^h(\Gamma)$ are sent to de Rham classes in $H^2_b(\Gamma)$.
Moreover, a de Rham quasimorphism is mapped to zero in $H^2_b(\Gamma)$ if and only if it is trivial (i.e., a homomorphism).

Recall that the Euler class in $H^2_b(\Gamma)$ is the image  under $\Psi_m$ of (a multiple of) the volume form associated to the metric $m$.
Since the volume form is closed but not exact, it follows that the Euler class does not lie in $\ker(\mathrm{comp}^2_\Gamma)= EH^2_b(\Gamma)$ (we refer the reader to \cite[Section 2.2]{Mar23} for more details).
On the other hand, every closed 2-form is the sum of an exact form and a multiple of the volume form.
This implies, using the commutativity of the diagram, that every de Rham class in $H^2_b(\Gamma)$ is the sum of a multiple of the Euler class and a class coming from a (de Rham) quasimorphism.

Since $\mcg(\Sigma)$ acts both on $Q^h(\Gamma)$ and $H^2_b(\Gamma)$, and the map $Q^h(\Gamma)\rightarrow H^2_b(\Gamma)$ is equivariant with respect to this action, we can now easily deduce Corollary \ref{cor: invariant continuous quasimorphismsintro} from Theorem \ref{thm: invariant continuous subspacesintro}.
If a finite-dimensional subspace $V\subseteq Q^h(\Gamma)$ generated by de Rham quasimorphisms is invariant under the action of $\aut(\Gamma)$, then its image $\overline{V}$ in $H^2_b(\Gamma)$ is also invariant. It follows from Theorem \ref{thm: invariant continuous subspacesintro} that $\overline{V}$ is either trivial or 1-dimensional, generated by the Euler class.
In conclusion, since the latter case is excluded, we deduce that $V$ consists of homomorphisms. 

In order to deduce Theorem \ref{thm: invariant_quasi_2}, we just need to make the following observations.
First of all, the space $\mathrm{Hom(\Gamma,\R)}$ is clearly $\aut(\Gamma)$-invariant.
Moreover, by Hurewicz theorem, we know that $\mathrm{Hom}(\Gamma,\R)\cong \R^{2g}$, where $g$ denotes the genus of $\Sigma$.
Therefore $V$ can be identified as a linear $\aut(\Gamma)$-invariant subspace of $\R^{2g}$ (and, hence, also $\mcg(\Sigma)$-invariant).
Since $\mcg(\Sigma)$ acts on $\R^{2g}$ by symplectic matrices and the representation $\mcg(\Sigma)\rightarrow \mathrm{Sp}(2g,\Z)$ is surjective \cite[Theorem 6.4]{farb2011primer}, Theorem \ref{thm: invariant_quasi_2} follows from the fact that the obvious action of $\mathrm{Sp}(2g,\Z)$ on $\R^{2g}$ is irreducible \cite[Theorem 2.8]{broaddus2011irreducible}, \cite[Proposition 3.2]{bor60}.

\bibliography{Biblio}
\bibliographystyle{fram_alpha}

\end{document}